\documentclass[11pt,reqno,a4paper]{amsart}
\usepackage[foot]{amsaddr}
\usepackage{amsthm,amsfonts,amsmath,amssymb,mathtools}
\usepackage[mathscr]{eucal}

\usepackage[T1]{fontenc}
\usepackage[utf8]{inputenc}

\usepackage{enumerate}
\usepackage{tikz-cd}
\usepackage{graphicx}

\numberwithin{equation}{section}

\theoremstyle{definition}
\newtheorem{definition}{Definition}

\theoremstyle{remark}
\newtheorem{remark}[definition]{Remark}

\theoremstyle{plain}
\newtheorem{theorem}[definition]{Theorem}
\newtheorem{result}[definition]{Result}
\newtheorem{lemma}[definition]{Lemma}

\newtheorem{corollary}[definition]{Corollary}

\newcommand{\bdy}{\partial}
\newcommand{\OM}{\Omega}
\newcommand{\disk}{\mathbb{D}}

\newcommand{\smoo}{\mathcal{C}}
\newcommand{\hol}{\mathscr{O}}
\newcommand{\holdom}{\mathscr{O}_\textsf{dom}}

\newcommand{\C}{\mathbb{C}}

\newcommand{\Z}{\mathbb{Z}}

\newcommand{\Ball}{\mathbb{B}}

\begin{document}
\title[Finiteness theorems for holomorphic mappings]{Finiteness theorems for
holomorphic mappings from products of hyperbolic Riemann surfaces}
\author{Divakaran Divakaran}
\address{Department of Mathematics, Institute of Mathematical Sciences, Chennai 600113, India}
\email{divakaran.divakaran@gmail.com}
\author{Jaikrishnan Janardhanan}
\address{Department of Mathematics, Indian Institute of Technology Madras, Chennai 600036, India}
\email{jaikrishnan@iitm.ac.in}


\begin{abstract}
  We prove that the space of dominant/non-constant holomorphic mappings from a
  product of hyperbolic Riemann surfaces of finite type into certain hyperbolic
  manifolds with universal cover a bounded domain is a finite set. 
  
\end{abstract}

\maketitle

\section{Introduction}
Complex analysis, in its most general setting, is the study of
holomorphic mappings between
complex spaces. Therefore, it is somewhat paradoxical that the space $\hol(X,Y)$ of 
holomorphic mappings between complex spaces $X$ and $Y$ often
comprises just constant mappings or contains only a finite/discrete set of
non-constant mappings. The classical theorems of Liouville and Picard illustrate
this phenomenon. In the context of compact Riemann surfaces, the famous Riemann--Hurwitz
formula puts severe restrictions on the comparative genera of the Riemann
surfaces $R$ and $S$ whenever $\hol_*(R,S)$ is non-empty. A somewhat deeper 
result is the following result of Imayoshi, which is an extension of a classical theorem 
due to de Franchis \cite{defranchis13}.

\begin{result}[Imayoshi {\cite[Theorem~2]{imayoshi83}}]\label{R:imayoshi}
  Let $R$ be a Riemann surface of finite type and let $S$ be a Riemann
  surface of finite type $(g,n)$ with $2g - 2 + n > 0$. Then $\hol_*(R,S)$ is a finite set.
\end{result}

A Riemann surface $S$ is said to be of \emph{finite type $(g,n)$} if it is
biholomorphic to a Riemann surface obtained by removing $n$ points from a
compact Riemann surface of genus $g$.
Observe that the condition  $2g - 2 + n > 0$
ensures that $S$ is a hyperbolic Riemann surface. This fact is crucial in the
proof. Also note that a hyperbolic Riemann surface is automatically complete
Kobayashi hyperbolic. However, aside from an appeal to Montel's
theorem for bounded domains, Kobayashi hyperbolicity plays \emph{no role} in
Imayoshi's proof. Rather, the proof relies mainly on the theory
of Fuchsian groups. It is well-known that the presence of (Kobayashi)
hyperbolicity in either the domain or target complex space
often forces the space of dominant holomorphic mappings to be
finite/discrete. See \cite{kobayashi75} and \cite{noguchi82} for examples of
such results. Also see \cite[Chapter~2]{zaldenberg89} for a detailed survey.

The following result of Shiga can be thought of as a higher dimensional
analogue of Result~\ref{R:imayoshi}.

\begin{result}[Shiga {\cite[Theorem~1.3]{shiga04}}]\label{R:shigagen}
  Let $X = \Ball^n/G$ be a complex hyperbolic manifold of divergence type
  and let $Y = \OM/\Gamma$ be a geometrically finite  $n$-dimensional complex 
  manifold where $\OM \subset \C^m$ is a bounded
  domain and $\Gamma$ is fixed-point-free discrete subgroup
  of $\textsf{Aut}(\OM)$. Suppose $G$ is finitely generated and that
  $\OM$ is complete with respect to the Kobayashi distance. Then
  $\holdom(X,Y)$ is a finite set. Furthermore, if the essential boundary
  dimension of $\OM$ is zero, then $\hol_*(X,Y)$ is a finite set.
\end{result}

\begin{remark}
  The proof of the above result uses methods inspired by the ones used by
  Imayoshi. However, as is evident from the hypotheses, Kobayashi hyperbolicity 
  plays a prominent role in the proof. 
\end{remark}

The purpose of this work is to both extend Result~\ref{R:imayoshi} and also
to clarify the role played by Kobayashi hyperbolicity. 
The central result of this paper is:

\begin{theorem}\label{T:fin} 
  Let $X := X_1 \times \dots \times X_n$ be a product of
  hyperbolic Riemann surfaces of finite type and let $Y = \OM/\Gamma$ be an
  $m$-dimensional complex manifold where $\OM \subset \C^m$ is a bounded
  domain and $\Gamma$ is fixed-point-free discrete subgroup
  of $\textsf{Aut}(\OM)$.
  \begin{enumerate}
    \item If $N$ is a tautly embedded complex submanifold of $Y$ then $\holdom
    (X,N)$ is a finite set.
    \item If $Y$ is geometrically finite and $\OM$ is complete hyperbolic then 
    $\holdom(X,Y)$ is a finite set.
    \item If in addition to the conditions in (2), the essential boundary
    dimension of $\OM$ is zero, then $\hol_*(X,Y)$ is a finite set.
  \end{enumerate}
  
\end{theorem}

\begin{remark}
  This theorem subsumes Result~\ref{R:imayoshi} (by Remark \ref{Rem:RSFTgeomfin}
  and (2) above) and can be viewed as an analogue
  of Result~\ref{R:shigagen} with $\Ball^n$ replaced by $\disk^n$. Our proof
  uses a combination of techniques used in the proofs of Result~\ref{R:imayoshi}
  and \ref{R:shigagen}. However, a key point of departure is our use 
  of a normal families argument (see Lemma~\ref{L:normal}) which simplifies 
  our proof. Our proof is also intentionally written in such a way that we get  
  new simplified proofs of Result~\ref{R:imayoshi} and Result~\ref{R:shigagen}
  \emph{mutatis mutandis}.

\end{remark}

\smallskip

\noindent \textbf{Notation:} We shall denote the space of holomorphic mappings
between two complex spaces $X$ and $Y$ by $\hol(X,Y)$. The space of non-constant
holomorphic mappings will be denoted by $\hol_*(X,Y)$ and dominant holomorphic
mappings by $\holdom(X,Y)$. $\disk$ shall denote the open unit disk in the
complex plane.

\smallskip

\noindent \textbf{Organization:} In Section~\ref{S:summary}, we summarize (with
references) the results and tools needed in the proof of Theorem~\ref{T:fin}.
The proof of Theorem~\ref{T:fin} forms the content of Section~\ref{S:main}.
Finally we present a few applications of our results in Section~\ref{S:app}.

\section{Summary of results and tools needed}\label{S:summary}

\subsection{Kobayashi hyperbolicity}

On each complex manifold $M$, one can define a pseudodistance $K_M$, called the
Kobayashi pseudodistance, that is
distance decreasing under holomorphic mappings, i.e., if $f:M \to N$ is a
holomorphic mapping between complex manifolds, then $K_M(x,y) \geq K_N(f(x),f
(y)) \ \forall x,y \in M$. By definition, the Kobayashi pseudodistance on
$\disk$ is nothing but the distance obtained by integrating the Poincaré
metric. See \cite{kobayashi67} for basic facts about 
the Kobayashi pseudodistance and  \cite{kobayashi1998,kobayashi05} for a
detailed exposition.  

We say that a complex manifold $M$ is \emph{hyperbolic} if $K_M$ is a distance.
If in addition $(M,K_M)$ is a complete metric
space, we say that $M$ is \emph{complete Kobayashi hyperbolic}. These notions
have been extensively studied in the literature and are indispensable in
the study of holomorphic mapping. 

The following result is often used to determine if a given complex manifold is 
(complete) hyperbolic.

\begin{result}\label{R:hypcomp}
  Let $M$ and $N$ be complex manifolds and $F:M \to N$ a covering map. Then $M$
  is (complete) hyperbolic iff $N$ is.
\end{result}

If $M$ and $N$ are hyperbolic complex manifolds then $\hol(M,N)$ is obviously an
equicontinuous family under the Kobayashi distance. It is a natural question
whether $\hol(M,N)$ is relatively compact as a subspace of $\mathcal{C}(M,N)$ in
the compact-open topology. As the compact-open topology coincides with
topology of uniform convergence on compacts, the following definition is
natural.

\begin{definition}
  Let $M$ and $N$ be complex manifolds. A subset $\mathcal{F} \subset \mathcal
  {C}(M,N)$ is said to be a \emph{normal family} if every sequence $\{f_n\}
  \subset \mathcal{F}$ has either a subsequence that converges uniformly on
  compacts to a function in $\mathcal{C}(M,N)$ or has a compactly divergent
  subsequence.

  A complex manifold $N$ is said to be \emph{taut} if for every complex
  manifold $M$ the set $\hol(M,N)$ is a normal family.
\end{definition}

The following result relates the notion of hyperbolicity and tautness.

\begin{result}[Kiernan {\cite[Proposition~3]{kiernan70}}]\label{R:taut}
  Complete hyperbolic complex manifolds are taut.
\end{result}

\subsection{Groups of divergence type}

The classical uniformization theorem states that the universal cover of a
Riemann surface is either $\C, \widehat{\C}$ or $\disk$. We say that a Riemann
surface is \emph{hyperbolic} if its universal cover is $\disk$. By 
Result~\ref{R:hypcomp} a hyperbolic Riemann surface is automatically complete
Kobayashi hyperbolic. 

If $R$ is a hyperbolic Riemann surface, we may write $R$
as $\disk/G$, where $G$ is a fixed-point free discrete subgroup of $\textsf
{Aut}(\disk)$. Such a group $G$ is knows as a \emph{Fuchsian group} and the
representation of $R$ as $\disk/G$ is known as the \emph{Fuchisan group
representation} of $R$. From the
theory of covering spaces, it also follows that $G \simeq \pi_1(R)$. If $R$ is 
a hyperbolic Riemann surface of finite type, we can get considerable information 
on the action of $G$ on $\disk$. We first need a definition.

\begin{definition}
  A Fuchsian group $G$ is said to be of \emph{divergence type} if 
  \[
    \sum_{g \in G} (1 - |g(z)|) = + \infty  \ \forall z \in \disk.
  \]
\end{definition}

\begin{result}\label{R:div}
  Let $R := \disk/G$ be a hyperbolic Riemann surface. Then the following are
  equivalent
  \begin{enumerate}
    \item $G$ is a Fuchsian group of divergence type.
    \item $R$ admits no Green's function.
    \item For almost every $\xi \in \bdy \disk$, we can find a sequence
    $g_\nu \in G$ such that $g_\nu \to \xi \ \forall z \in \disk$
    uniformly and non-tangentially on compacts.
  \end{enumerate}
\end{result}

\begin{remark}
  It follows from the above result that hyperbolic Riemann surfaces of finite
  type can be represented as $\disk/G$ where $G$ is a Fuchsuan group of
  divergence type. See \cite{tsuji75} for a proof of the above result and
  related results.
\end{remark}

\subsection{The notion of ends and essential boundary dimension}

\begin{definition}
  Let $X$ be a topological manifold. An \emph{end} of $X$ is a decreasing
  sequence of connected open sets
  \[
    U_1 \supset U_2 \supset \dots,
  \]
  with the property that given any compact set $K \subset X$, there is an $N_0
  \in \Z_+$ such that $K \cap U_{N_0} = \emptyset$. Two ends $U_1 \subset
  U_2 \subset \dots$ and $U_1' \subset U_2' \subset \dots$ are considered
  equivalent if for each $n \in \Z_+$, we can find $N \in \Z_+$ such that $U_n
  \subset U_N'$, and vice-versa.
\end{definition}

\begin{remark}
  The notion of ends can be used to construct a nice compactification of $X$.
  The space of ends also allows one to classify non-compact surfaces. See 
  \cite{richards63}.
\end{remark}

In this work, we are mainly interested in complex manifolds with finitely many
ends of a certain kind known as parabolic ends. We first define the notion of a
pluripolar set.

\begin{definition}
  A set $E \subset \C^n$ is said to be a \emph{complete pluripolar set} if for
  some plurisubharmonic function $u: \C^n \to [-\infty,\infty), E = u^{-1}
  \{-\infty\}$.
\end{definition}

We now define what it means for a manifold $Y := \OM/\Gamma$ to be geometrically
finite.

\begin{definition}

  Let $Y := \OM/\Gamma$ be a complex manifold, where $\OM \subset \C^n$ is a
  bounded domain and $\Gamma$ is a discrete and fixed-point free subgroup of
  $\textsf{Aut}(\OM)$. Denote the covering map by $\pi: \OM \to X$. We say that
  the end 
  \[
     U_1 \supset U_2 \dots,
  \] 
  is a \emph{parabolic end} if for some $N_0$, we can find a connected component of 
  $\pi^{-1}(U_0)$, say $\widetilde{U}_{N_0}$, such that $\overline{\widetilde
  {U}}_{N_0} \cap \bdy \OM \subset \bigcup_{j = 1}^\infty R_j$ where each $R_j
  $ is a complete pluripolar set in $\C^n \setminus \OM$.

  We say that $X$ is \emph{geometrically finite} if $X$ has only finitely many
  ends and each end is a parabolic end.  
\end{definition}

  \begin{remark}
    Any hyperbolic Riemann surface of finite type is automatically geometrically
    finite.
    \label{Rem:RSFTgeomfin}
  \end{remark}

  \begin{remark}
    If a manifold $X$ has only finitely many inequivalent ends then we may, and
    shall do so tacitly, represent each of these finitely many ends by finitely
    many pairwise disjoint connected open sets.
  \end{remark}

\subsection{Fatou--Riesz theorems on the polydisk}

We will need the following version of the Fatou--Riesz theorem for the polydisk.

\begin{result}[Tsuji {\cite{tsuji45}}]\label{R:fatou}
  Let $f$ be a bounded holomorphic function on $\disk^n$. Then for almost every
  point $\xi$ on the distinguished boundary (the torus $\mathbb{T}^n$)
  \[
    \lim_{z \to \xi} f(z),
  \]
  exists whenever $z \to \xi$ non-tangentially. We 
  denote this limit by
  $f^*(\xi)$. Furthermore, if $g$ is another bounded holomorphic function on
  $\disk^n$ and $f^*(\xi) = g^*(\xi)$ on a set of positive measure on $\mathbb
  {T}^n$, then $f \equiv g$.
\end{result}

The following lemma has been proven for an arbitrary bounded $\smoo^2$-smooth
domains by Shiga (see \cite[Theorem~3.1]{shiga04}). We can view it as a
generalization of the last part of the previous theorem. As the proof given by
Shiga relies only on the existence of a Poisson kernel, his proof works for the
polydisk also.
\begin{lemma}\label{L:uniq}
  Let $\phi:\disk^n \to \C^m $ be a non-constant bounded holomorphic mapping and
  let $E$ be a complete pluripolar set in $\C^m$.  If $\phi^*(\xi) \in E$
  for a set of positive measure on $\mathbb{T}^n$, then $\phi(\disk^n) \subset E$.
\end{lemma}






We will also need the following lemma which shows that under the hypothesis of
Theorem~\ref{T:fin}, the non-tangential limits of the lift of a non-constant map lie on
the boundary. The proof is essentially the same as the proof of 
\cite[Lemma~2.2]{shiga04}. 

\begin{lemma}\label{L:proper}
  Let $X = X_1 \times \dots \times X_n$ be a product of hyperbolic Riemann
  surfaces of finite type and let $Y := \OM/\Gamma$ be a complex manifold, where 
  $\OM \subset \C^n$ is a bounded domain and $\Gamma$ is a discrete and 
  fixed-point free subgroup of $\textsf{Aut}(\OM)$. Let $f : X \to Y$ be a
  non-constant holomorphic map and let $F : \disk^n \to \OM$ be any lift. Then
  for almost every $\xi \in \mathbb{T}^n, F^*(\xi) \in \bdy \OM$.
\end{lemma}






We require one more definition.

\begin{definition}
  Let $D \subset \C^n$ be a bounded domain. We say that $D$ has \emph{essential
  boundary dimension zero}, if there exists a family $\{R_j\}_{j=1}^\infty$ of
  complete pluripolar sets with $R_j \cap \bdy D \neq \emptyset \ \forall j$
  such that $\bdy D \setminus \bigcup_{j=1}^\infty R_j$ contains no analtic
  space of positive dimension. 
\end{definition}

\section{Proofs of our main results}\label{S:main}

  In this section, we shall give the proof of our main theorem
  (Theorem~\ref{T:fin}). We first  prove a rigidity theorem for holomorphic
  mappings (Theorem~\ref{T:rigidity}). This rigidity theorem is an extension of
  a rigdity result used by Imayoshi to prove Result~\ref{R:imayoshi}.

  \begin{theorem} \label{T:rigidity} Let $X := X_1 \times \dots \times X_n$ be
    a product of hyperbolic Riemann surfaces of finite type and let $Y = \OM/\Gamma$
    be an $m$-dimensional complex manifold where $\OM \subset \C^m$ is a
    bounded domain and $\Gamma$ is fixed-point-free discrete
    subgroup of $\textsf{Aut}(\OM)$. Let $\phi, \psi :X \to Y$ be
    holomorphic mappings such that $\phi_* = \psi_*$. Then 
    \begin{enumerate}
      \item If $\phi$ (or $\psi$) is dominant, then $\phi = \psi$.
      \item  If $\phi$ (or $\psi$) is non-constant and $\OM$ has essential
      boundary dimension zero, then $\phi = \psi$.
    \end{enumerate}
  \end{theorem}
  \begin{remark}
    The above theorem is also true when $X = \Ball^n/G$ (see Section~\ref
    {S:app}). Our proof works \emph{mutatis mutandis}. See also \cite[Theorem~1.1]{shiga04}.
  \end{remark}
    
  \begin{proof} Let $G_i$ be the Fuchsian group of divergence type such that
    $X_i =
    \disk/G_i$. Then $X = \disk^n/G$ where $G := \bigoplus_{i=1}^n G_i$. Let
    $\widetilde{\phi}, \widetilde{\psi} : \disk^n \to \OM$ be lifts of
    $\phi$ and $\psi$, respectively. As $\disk^n$ is simply-connected, the group
    $G$ and $\pi_1(X)$ are isomorphic.  Fix a point $y\in Y$ and a point $\tilde{y}\in \OM$ that gets mapped to $y$ under the quotient map.  Then, given an element $[\delta] \in \pi_1(Y,y)$ there is an element
    $\gamma \in \Gamma$ that takes $\tilde{y}$ to the end point of the lift of $\delta$ starting at $\tilde{y}$.  This element is independent of the choice of
    representative of $[\delta]$. It is easy to
    check that, in this way, we get a homomorphism from $\pi_1(Y)$ to
    $\Gamma$. Thus, given any map from $X$ to $Y$, by composing, we get a
    homomorphism from $G$ to $\Gamma$. The hypothesis that $\phi_* =
    \psi_*$ implies that both mappings $\phi$ and $\psi$ induce the same
    homomorphism from $G$ to $\Gamma$. Let $\chi$ be this induced homomorphism.

    It follows that
    \begin{align} \label{E:hom1}
      \widetilde{\phi} \circ g &= \chi(g) \circ \widetilde{\phi}\\
      \widetilde{\psi} \circ g &= \chi(g) \circ \widetilde{\psi},
      \label{E:hom2}
    \end{align} for each $g \in G$.

    As each $G_i$ is of divergence type, by Result~\ref{R:div}, for almost 
    every $\xi \in \mathbb{T}^n$, we
    can find a sequence $g_{i,k} \in G_i$ such that $g_{i,k}(z) \to \xi_i$ non-tangentially
    on compact subsets of $\disk$. Let $g_k := (g_{1,k},
    g_{2,k},\dots,g_{n,k})$. Then $g_k(z) \to \xi$ non-tangentially and
    uniformly on compact subsets of $\disk^n$.

    Now, as $\OM$ is a bounded domain, Fatou's theorem for the polydisk 
    (Result~\ref{R:fatou}) implies that for almost every $\xi
    \in \mathbb{T}^n$, each component function of $\widetilde{\phi}$ and
    $\widetilde{\psi}$ converge as $z \to \xi$ non-tangentially. We define
    $\tau(\xi)$ and $\eta(\xi)$ to be the non-tangential limits of
    $\widetilde{\phi}$ and $\widetilde{\psi}$, respectively, as $z \to \xi$
    (the functions $\tau$ and $\psi$ are defined almost-everywhere on $\mathbb
    {T}^n$). This, in combination with arguments in the previous paragraph,
    shows that for almost every $\xi \in \mathbb{T}^n$, we can find a sequence
    $g_k \in G$ such that
    \begin{align*}
      \lim_{k \to \infty} \widetilde{\phi} \circ g_k(z) &= \tau(\xi) \\
      \lim_{k \to \infty} \widetilde{\psi} \circ g_k(z) &= \eta(\xi),
    \end{align*} 
    for each $z \in \disk^n$. As each $\chi(g_k) \in
    \textsf{Aut}({\OM})$, and as $\OM$ is bounded, we may use Montel's theorem
    and assume without loss of generality that $\chi(g_k)$ converges uniformly
    on compact subsets of $\OM$ to a holomorphic map $A_\xi: \OM \to \C^n$.
    Using
    \eqref{E:hom1} and \eqref{E:hom2}, we have $A_\xi \circ \widetilde{\phi}
    (z) = \tau(\xi)$ and $A_\xi  \circ \widetilde{\psi}(z) = \eta(\xi)$,
    whenever $z \in \disk^n$. 

    Now assume that $\phi$ is dominant. This implies that $\widetilde{\phi}$
    is dominant as well. Therefore, we can find a polydisk $P_\xi \Subset
    \disk^n$  such that $\widetilde{\phi}(P_\xi)$ has an interior point. This
    forces $A_\xi$ to be a constant map by the identity principle. Hence
    $\tau$ and $\eta$ agree almost everywhere on the torus which means that 
    $\widetilde{\phi} = \widetilde{\psi}$ (by Fatou's theorem) and 
    hence $\phi = \psi$ as claimed.

    Now assume that the essential boundary dimension of $\OM$ is $0$ and that
    $\phi$ is non-constant.
    This implies that $\widetilde{\phi}$ is non-constant as well. 
    Each $\chi(g_k)$ is an automorphism of the bounded domain $\OM$ that
    converges in the compact-open topology to the map $A_\xi: \OM \to \overline
    {\OM}$. Consequently, by Cartan's theorem for biholomorphisms, either
    $A_\xi$ is
    an automorphism or the image of $A_\xi$ is fully contained in $\bdy \OM$.
    However, as $\widetilde{\phi}$ is non-constant, $\widetilde{\phi}(P_\xi)$
    cannot be a singleton set (by the identity principle), where $P_\xi \Subset
    \disk^n$ is any polydisk. Hence $A_\xi$ is not injective as 
    $A_\xi  \circ \widetilde{\phi}(z)$ is constant and consequently $A_\xi
    (\OM) \subset \bdy \OM$. 

    We have shown that for almost every $\xi \in \mathbb{T}^n$, we have $A_\xi
    (\OM) \subset \bdy \OM$. We claim that for almost every $\xi \in \mathbb
    {T}^n$, $A_\xi$ is constant. Suppose not. 
    Since the essential boundary dimension of $\OM$ is $0$, we can find 
    countably many complete 
    pluripolar sets $R_j$ of $\C^n \setminus \OM$ such that the image of any
    non-constant holomorphic function into $\bdy \OM$ must lie fully in
    $\bigcup_{j=1}^\infty R_j$. Therefore, we can find a set $E \subset
    \mathbb{T}^n$ of positive measure with the property that $A_\xi(\OM)
    \subset R_{j_0}$ for some $j_0 \in \Z_+$ and $\forall \xi \in E$. This means
    that for a set of of positive measure, the non-tangential limits of
    $\widetilde
    {\phi}$ lie in a complete pluripolar set and consequently $\widetilde{\phi}
    (\disk^n)$ must be entirely contained in the same pluripolar set by
    Lemma~\ref{L:uniq}. This is not possible. Consequently,
    $A_\xi$ is a constant for almost every $\xi \in \mathbb{T}^n$. This means
    that the non-tangential limits of $\widetilde{\phi}$ and $\widetilde{\psi}$
    coincide
    almost everywhere on $\mathbb{T}^n$ proving that $\widetilde{\phi} =
    \widetilde{\psi}$ by Fatou's theorem. Hence $\phi = \psi$.

  \end{proof}

  The next lemma clarifies the role of geometrical finiteness of $\OM$.

  \begin{lemma}\label{L:normal}
    Let $X := X_1 \times \dots \times X_n$ be
    a product of hyperbolic Riemann surfaces of finite type and let $Y = \OM/\Gamma$
    be an $m$-dimensional geometrically finite complex manifold where $\OM
    \subset \C^m$ is a bounded complete hyperbolic domain and $\Gamma$ is fixed-point-free discrete
    subgroup of $\textsf{Aut}(\OM)$.
    Then any sequence $\{f_n \in \hol_*(X,Y)\}$
    converges in the compact-open topology to a map in $\hol(X,Y)$.
  \end{lemma}

  \begin{proof}
    As in the proof of Theorem~\ref{T:rigidity}, we write $X$ as $\disk^n/G$.

    By Result~\ref{R:taut}, complete hyperbolic manifolds are taut. Therefore any
    sequence $\{f_\nu \in \hol_*(X,Y)\}$ must be either have a compactly
    convergent or compactly divergent subsequence. 

    Suppose there is a sequence $\{f_\nu \in \hol_*(X,Y)\}$ that is compactly
    divergent.  We 
    view each $X_i$ as being an open subset of a compact Riemann surface $R_i$. 
    For each point $x \in R_i \setminus X_i$, let $D_x$ be a small disk in $R_i$
    with $x$ as the centre, chosen in such a way that if $x_1,x_2 \in R_i 
    \setminus X_i$ are distinct
    points, then the closures of $D_{x_1}$ and $D_{x_2}$ are disjoint. Let $K
    := \bigtimes_{i = 1}^n R_i \setminus \bigcup_{x \in R_i \setminus X_i}
    D_x$. As $K$ is a deformation retract of $X$, $\pi_1(K) = \pi_1(X)$. Any
    curve $\gamma \subset X$ can be homotoped to a curve that lies entirely in
    $K$. For large $\nu$,
    by compact divergence and the fact that $K$ is connected, we may assume that
    $f_\nu(K) \subset V$ where $V \subset Y$ is one of the finitely many
    parabolic ends of $Y$. Let $V'$ be a connected component of a pre-image
    of $V$ under the covering map such
    that $\overline{V'} \cap \bdy \OM \subset \bigcup_{i = 1}^\infty R_i$ where
    each $R_j$ is a complete pluripolar set in $\C^m \setminus \OM$. 

    Fix $\nu$ large and let $f := f_\nu$ and let $\xi: G \to \Gamma$ be the
    induced homomorphism. Denote by $F$ the lift of $f$ with the property
    that $F(0) \in V'$. Let $p_0 \in M$ be the image of $0$
    under the covering map. We may assume that $p_0 \in K$. Let $g \in G$ and
    let $L$ be a loop based at $p_0$ contained entirely in $K$ that corresponds
    to $g$. Denote the lift of $L$ based at $0$ by $\widetilde{L}$. Then $F
    (\widetilde{L}) \subset V'$. As $F(g(0)) = \chi(g)(F(0))$ and the endpoint
    of
    $\widetilde{L}$ is $g(0)$, it follows that $\chi(g)(F(0)) \in V'$.
    For almost every $\xi \in \mathbb{T}^n$, we can find a sequence $g_\nu \in
    G$
    such that $g_\nu(0) \to \xi$ non-tangentially. We assume that $F^*(\xi)$
    exists
    and that $F^*(\xi) \in \bdy \OM$. But $F(g_\nu(0)) = \chi(g_\nu)
    (F(0)) \in V'$. Hence $F^*(\xi) \in \overline{V'} \cap \bdy \OM \subset
    \bigcup_{i = 1}^\infty R_i$. This force $F$ to be a constant by 
    Lemma~\ref{L:uniq}.
  \end{proof}
  
  \noindent \textbf{Proof of Theorem~\ref{T:fin}:} Let $\{ f_k \} \subseteq
    \holdom(X,N)$ be a sequence of distinct dominant holomorphic mappings. As
    $N$ is tautly embedded in $Y$, we may assume that the sequence $\{f_k\}$
    converges in the compact-open topology to a map $f: X \to Y$.

    As in the proof of Lemma~\ref{L:normal}, we can find a connected and
    compact set $K$ such that finitely many closed loops contained in $K$
    generate $\pi_1(X)$. Now, as $\{f_k\}$ converges uniformly on $K$ to $f$, it
    is easy to see that for suitably large $k$, the $f_k(\gamma)$ and $f
    (\gamma)$ are homotopic where $\gamma$ is curve whose image lies entirely in
    $K$. Thus all the $f_k$'s induce the same
    homomorphism on $\pi_1(X)$. Consequently, by Theorem~\ref{T:rigidity}, for
    large $k$, all the $f_k$'s are equal, which is a contradiction.

    We prove the
    last part. Suppose $\{f_k \} \subseteq \hol_*(X,Y)$ is a sequence of
    distinct holomorphic mappings. By Lemma~\ref{L:normal}, we may, by passing
    to a subsequence, if needed, assume that the $f_k$'s converge uniformly on
    compacts. The previous argument now delivers the proof.
    
    Part 2 of this theorem follows by similar arguments.
  \qed

  \section{Applications}\label{S:app}

  We will now highlight a few applications of our results. The first application
  is rather trivial but nevertheless interesting.

  \begin{corollary}
    Let $X := X_1 \times \dots \times X_n$ be
    a product of hyperbolic Riemann surfaces of finite type and let $Y = \OM/\Gamma$
    be an $m$-dimensional compact complex manifold where $\OM
    \subset \C^m$ is a bounded domain and $\Gamma$ is fixed-point-free discrete
    subgroup of $\textsf{Aut}(\OM)$ Then $\holdom(X,N)$
    is a finite set where $N \subset Y$ is any complex submanifold (including
    $Y$ itself).
  \end{corollary}

  \begin{proof}
    As $Y$ is compact hyperbolic, it is complete hyperbolic and therefore taut.
    Any sequence of holomorphic maps $f_k : X \to Y$ cannot have a
    compactly divergent subseqeunce as $Y$ is compact. Thus $N$ is tautly
    embedded in $Y$ and the corollary follows from Theorem~\ref{T:fin}.
  \end{proof}
  
  Our next application is an extension of Theorem~\ref{R:imayoshi} to maps
  between products of Riemann surfaces.

  \begin{corollary}
  	Let $X = X_1 \times \dots \times X_n$ be a product of hyperbolic Riemann
    surfaces of
  	finite type and let $Y = Y_1 \times \dots \times Y_m$ be a product of
  	hyperbolic Riemann surfaces each of which can be embedded inside a
  	compact Riemman surface. Then $\hol_\textsf{dom}(X,Y)$ is a finite set.
  \end{corollary}

  \begin{remark}
    In \cite{janardhanan14}, finite proper holomorphic mappings between
    products of certain hyperbolic Riemann surfaces was studied. The above
    result is a partial extension of the main result there.
  \end{remark}

  \begin{proof}
    We view each $Y_i$ as an open subset of a compact Riemann surface $S_i$.
    Now, $Y_i$ is obtained by excising a set $E_i$ from $S_i$. Choose finite
    sets  $E_i' \subset E_i$ such that $S_i \setminus E_i'$ is hyperbolic. Let
    $N
    := S_1 \setminus E_1' \times \dots \times S_n \setminus E_n'$. Then $N$ is a
    product of hyperbolic Riemann surfaces of finite type. It suffices to show
    that
    $\holdom(X,N)$ is a finite set. But any hyperbolic Riemann surface of finite
    type is geometrically finite. It follows from Theorem~\ref{T:fin} that 
    $\holdom(X,N)$ is a finite set.
  \end{proof}

  We end with an observation about our results. We used three features of the
  representation of the manifold $X$ as $\disk^n/G$ in our results:

  \begin{enumerate}
    \item We have a Fatou--Riesz theorem on $\disk^n$.
    \item For almost every element $\xi \in \mathbb{T}^n$, we have a
    sequence of elements $g_k \in G$ such that $g_k$ converges uniformly and
    non-tangentially on compacts to $\xi$. In other words, $G$ is of divergence
    type.
    \item $G$ (or equivalently $\pi_1(X)$) is finitely generated.
  \end{enumerate}
  
  In view of the above observation, our proof of Theorem~\ref{T:fin}
  works \emph{mutatis mutandis} when $X$ is of the form
  $\Ball^n/G$
  where $G$ is a finitely generated subgroup of $\textsf{Aut}(\Ball^n)$ of 
  divergence type. See \cite{shiga04} for a precise definition of divergence
  type. We can also replace $X$ by $D/G$ where $D$ is any bounded domain and
  group $G \subset \textsf
  {Aut}(D)$ that have the analogues of the three features mentioned above. For
  instance, in the statement of Theorem~\ref{T:fin}, we may replace each
  Riemann surface $X_i$ by a quotient $\Ball^n/G_i$ where each $G_i \subset
  \textsf{Aut}(\Ball^n)$ is a finitely generated subgroup of divergence type. 
  
\bibliographystyle{amsalpha} \bibliography{finiteness}

\end{document}